\documentclass[10pt]{amsproc}
\usepackage{amssymb}
\usepackage{graphicx,color}
\usepackage{amsmath}
\usepackage{amscd} 

\theoremstyle{plain}
 \newtheorem{thm}{Theorem}[section]
 
 \newtheorem{lem}{Lemma}[section]
 \newtheorem{cor}{Corollary}[section]
\theoremstyle{definition}
 
 \newtheorem{dfn}{Definition}[section]
\theoremstyle{remark}
 \newtheorem{rem}{Remark}[section] 
 \numberwithin{equation}{section}

\renewcommand{\leq}{\leqslant}
\renewcommand{\geq}{\geqslant}

\title[products of two Cantor sets]{PRODUCTS OF TWO CANTOR SETS}

\author[Y.\ Takahashi]{\bfseries YUKI TAKAHASHI}

\address{Department of Mathematics, University of California, Irvine, CA~92697, USA}

\email{takahasy@math.uci.edu}

\thanks{Y.\ T. \ was supported in part by NSF grant DMS-1301515 (PI: A.\ Gorodetski).}

\date{today}

\begin{document}

\vspace{18mm}
\setcounter{page}{1}
\thispagestyle{empty}

\begin{abstract}
We consider products of two Cantor sets, and obtain the optimal
estimates in terms of their thickness that guarantee that their product
is an interval. This problem is motivated by the fact that the spectrum of the Labyrinth model, which is a 
two dimensional quasicrystal model, is given by a product of two Cantor sets. We also discuss the connection 
with the question on the structure of intersections of two Cantor sets which was considered by many authors previously.
\end{abstract}
\maketitle

\section{Introduction}  
\subsection{Sums and products of Cantor sets}

Sums of Cantor sets have been considered in many papers and in many different settings 
(e.g., \cite{Anisca}, \cite{BMPM}, \cite{DGsurvey}, \cite{DG11}, \cite{DG13}, \cite{Droglu}, \cite{Hall}, \cite{Hochman}, \cite{Hochman2}, \cite{Moreira}, \cite{Moreira2}, \cite{Moreira3}, \cite{PalisTakens}, \cite{Solomyak}, \cite{Solomyak1997}). 
It arises naturally in dynamical systems in the study of homoclinic bifurcations \cite{PalisTakens}.  
It also arises in number theory in connection with continued fractions as initiated by Hall \cite{Hall}. 
In \cite{Hall}, the author proved that any real number can be 
written as a sum of two real numbers whose continued fractional coefficients are at most $4$. 
It is also connected to spectral theory (e.g., \cite{DGsurvey}, \cite{DG11}, \cite{DG13}, \cite{Lifshitz}). The spectra of certain types of 
two dimensional quasicrystal models can be written as sums of two dynamically defined Cantor sets \cite{DG11}. 
The study of sums of Cantor sets also has natural connection to the study of intersections of Cantor sets 
(e.g., \cite{Honary}, \cite{Hunt}, \cite{Kraft}, \cite{Kraft3}, \cite{Moreira0}, \cite{Newhouse}, \cite{Williams}). 

In \cite{Newhouse}, Newhouse proved the following so-called Gap Lemma 
(for the definition of thickness $\tau(\cdot)$, see section \ref{765}):

\begin{lem}[Gap Lemma]
Let $K$, $L$ be Cantor sets with $\tau(K) \cdot \tau(L) > 1$. Then, 
if neither $K$ nor $L$ lies in a complementary domain of the other, $K \cap L$ contains at least one element. 
\end{lem}

In fact, if Newhouse's proof is slightly altered the condition $\tau(K) \cdot \tau(L) > 1$ may be replaced with $\tau(K) \cdot \tau(L) \geq 1$. 
The following is a direct consequence of Gap Lemma:

\begin{thm}\label{sum_sum_sum}
Suppose $K$ and $L$ are Cantor sets with $\tau(K) \cdot \tau(L) \geq 1$. Assume also that the size of the largest 
gap of $K$ is not greater than the diameter of $L$, and the size of the largest gap of $L$ is not greater than the 
diameter of $K$. Then $K + L$ is a closed interval. 
\end{thm}

Using Theorem \ref{sum_sum_sum} as a tool, 
we consider products of two Cantor sets. This problem arises naturally in the study of  the 
spectrum of the Labyrinth model \cite{Takahashi}. 
For any two Cantor sets $K, L > 0$, we have 
\begin{equation*}
K \cdot L = \exp \left( \log K + \log L \right).
\end{equation*} 
Using this equality, if $K$ and $L$ do not contain $0$, some results of products of Cantor sets can be immediately obtained by 
that of sums of Cantor sets. For example, in \cite{Astels}, the authors obtained an estimate for products of two or more Cantor sets 
to be an interval. 

The main difficulty arises when $K$ or $L$ contain $0$. 
To the best of our knowledge, this case has never been discussed before. 
If $K$ contains $0$, 
then $\log K$ is ``stretched to negative infinity", making products of Cantor sets different from 
sums of Cantor sets. 
Indeed, for example, in section \ref{765} we will show that under the condition of $\tau(K) \cdot \tau(L) > 1$,
 products of two Cantor sets 
$K \cdot L$ may contain countably many disjoint closed intervals. 
This is a phenomena which never appears in sums of two Cantor sets under the condition of 
$\tau(K) \cdot \tau(L) \geq 1$. 

Initially this work was motivated by the question on spectral properties of the Labyrinth model \cite{Takahashi}. 
As mentioned above, the spectrum of the Labyrinth model is a product of two Cantor sets, and in fact, these two Cantor sets both 
contain the origin. Using our results, we can show that the spectrum of the Labyrinth model is an interval   
for the small coupling constant regime. See \cite{Takahashi}.

\subsection{Main results}
For any Cantor set $K$, we denote $K \cap ( 0, \infty )$ by $K_+$, and 
$-\left( K \cap ( -\infty, 0 ) \right)$ by $K_-$. Let us give the following definition:
\begin{dfn}
Let $K$ be a Cantor set. We call $K$ a 
\begin{itemize}
\item[(1)] \emph{$0$-Cantor set} if $K_{+}, K_{-} \neq \phi$, $\inf K_+ = 0$, and $\inf K_- = 0$; 
\item[(2)] \emph{$0^{+}$-Cantor set} if $\min K = 0$; 
\item[(3)]  \emph{$0^{\times}$-Cantor set} if $K_{+}, K_{-} \neq \phi$, and $0 \notin K$; 
\item[(4)] \emph{$0^+_-$-Cantor set} if $K_{+}, K_{-} \neq \phi$, $\inf K_+ = 0$, and $\inf K_- > 0$. 
\end{itemize}
\end{dfn}
Our main results are the following:
\begin{thm}\label{kamidaro}
Let $K, L$ be $0^+$-Cantor sets. 
Then, $K \cdot L$ is an interval if 
\begin{equation}\label{465}
\tau(L) \geq \frac{2 \tau(K) + 1 }{ \tau(K)^2 }, \text{ or } \ \tau(K) \geq \frac{2 \tau(L) + 1 }{ \tau(L)^2 }.
\end{equation}
In particular, if 
\begin{equation*}
\tau(K) = \tau(L) \geq \frac{1 + \sqrt{5}}{2}, 
\end{equation*}
then $K \cdot L$ is an interval. 
Furthermore, let $M, N > 0$ be real numbers with 
\begin{equation}\label{thm0}
N < \frac{2M + 1}{M^2}, \text{ \ and \ } M < \frac{2N + 1}{N^2}.  
\end{equation}
Then 
\begin{itemize}
\item[(1)] there exist $0^+$-Cantor sets $K, L,$ such that $\tau(K) = M$, $\tau(L) = N$,  
and $K \cdot L$ is a disjoint union of $\{0\}$ and countably many closed intervals; 
\item[(2)] for any $k \geq 2$, there exist $0^+$-Cantor sets $K, L,$ such that $\tau(K) = M$, $\tau(L) = N$, and 
$K \cdot L$ is a disjoint union of $k$ closed intervals. 
\end{itemize}
\end{thm}

Similarly, we have 
\begin{thm}\label{goddaro}
Let $K$ be a $0^+$-Cantor set and let $L$ be a $0$-Cantor set. 
Then, $K \cdot L$ is an interval if 
\begin{equation}\label{3654}
\tau(L) \geq \frac{2 \tau(K) + 1}{\tau(K)^{2}}.
\end{equation}
Furthermore, let $M, N > 0$ be real numbers with 
\begin{equation}\label{thm2}
N < \frac{2M + 1}{M^2}. 
\end{equation}
Then
\begin{itemize}
\item[(1)] there exists a $0^+$-Cantor set $K$ and a $0$-Cantor set $L$ such that $\tau(K) = M$, $\tau(L) = N$, and 
$K \cdot L$ is a disjoint union of countably many closed intervals; 
\item[(2)] for any $k \geq 2$, 
there exists a $0^+$-Cantor set $K$ and a $0$-Cantor set $L$ such that $\tau(K) = M$, $\tau(L) = N$, 
and $K \cdot L$ is a disjoint union of $k$ closed intervals. 
\end{itemize}
\end{thm}

We also have the following:
\begin{thm}\label{migotoda}
Let $K, L$ be $0$-Cantor sets. Then, if
\begin{equation}\label{46578}
2 \left( \tau(K) + 1 \right) \left( \tau(L) + 1 \right) \leq \left( \tau(K) \tau(L) - 1 \right)^{2},
\end{equation}
$K \cdot L$ is an interval. In particular, if 
\begin{equation*}
\tau(K) = \tau(L) \geq 1 + \sqrt{2},
\end{equation*}
then $K \cdot L$ is an interval. 
Furthermore, let $M, N > 0$ be real numbers with 
\begin{equation}\label{thm3}
2 (M + 1)(N + 1) > (MN - 1)^2. 
\end{equation}
Then, there exist $0$-Cantor sets $K, L,$ such that $\tau(K) = M$, $\tau(L) = N$, and 
$K \cdot L$ is a disjoint union of two intervals.
\end{thm}

\begin{centering}
\begin{figure}[t]
\includegraphics[scale=1.00]{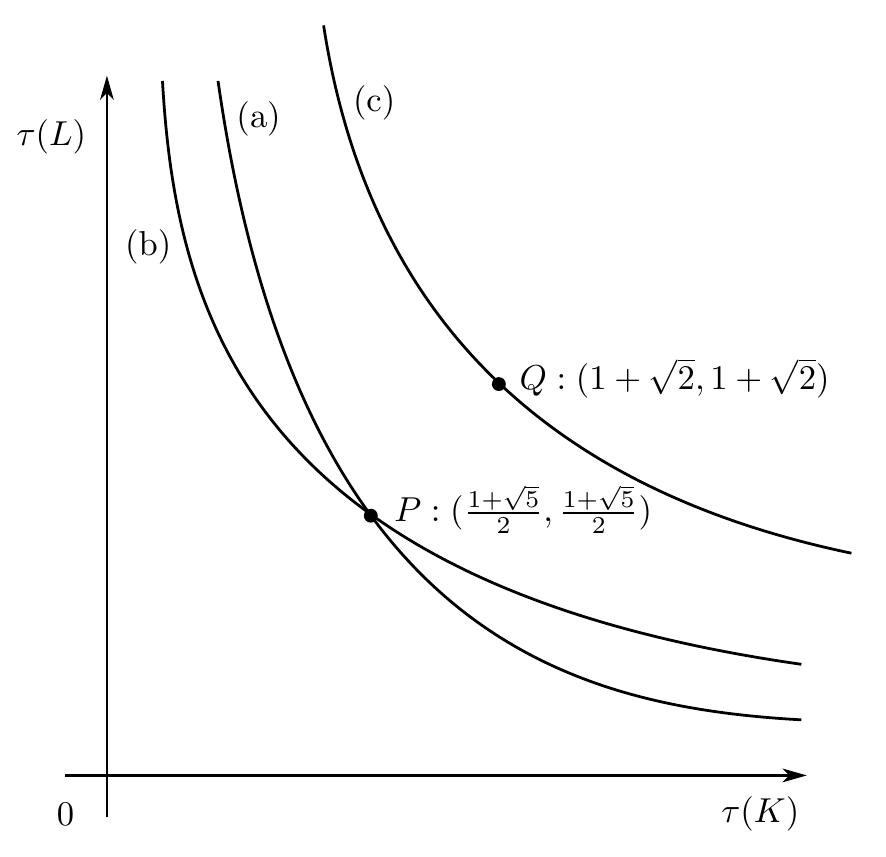}
\caption{(a), (b) are the graphs of (\ref{465}), and (c) is the graph of (\ref{46578}).}
\label{figure1}
\end{figure}
\end{centering}

We believe that Theorem \ref{migotoda} does not hold if we replace ``disjoint union of two intervals" with  
``disjoint union of countably many closed intervals", 
or  ``disjoint union of $k$ ($\geq 3$) closed intervals". 
It would be interesting to prove that this is indeed true. 

To ensure that the product may contain countably many disjoint closed intervals, we have the following estimate:

\begin{thm}\label{43215}
Let $M, N > 0$ be real numbers with $M \geq N$. 
Then, if 
\begin{equation}\label{intersection}
M < \frac{N^2 + 3N + 1}{N^2}, \text{ or } \ N < \frac{(2M+1)^2}{M^3},  
\end{equation}
there exist $0$-Cantor sets $K, L,$ such that $\tau(K) = M$, $\tau(L) = N$, and 
$K \cdot L$ is a disjoint union of $\{0\}$ and countably many closed intervals.
\end{thm}
We are not sure whether the estimate in Theorem \ref{43215} is optimal. We leave this as an open problem.

As stated above, the study of intersections of Cantor sets is naturally connected to the study of sums of Cantor sets. 
In fact, it is also connected to the study of products of Cantor sets. 
Indeed, the estimate in Theorem \ref{43215} is exactly the same as the estimate that appears in \cite{Hunt} and \cite{Kraft}, 
which is the optimal estimate that two interleaved Cantor sets may have a one point intersection
(\cite{Hunt} and \cite{Kraft} obtained the same results independently).
Our method provides different proofs to some of the results presented there. 
On the other hand, we do not believe that our results follow from 
the techniques in \cite{Hunt}, \cite{Kraft}. 
We discuss this connection in more detail in section \ref{65743}. 

\subsection{Structure of this paper}
In section \ref{765}, we give necessary definitions and prove that the product of two Cantor sets $K \cdot L$ may contain 
countably many disjoint closed intervals in the case of 
$\tau(K) \cdot \tau(L) > 1$.
In section \ref{7689} we prove the key lemma, which is the optimal estimate of the thickness of $\log K$ 
for certain type of Cantor set $K$.
Using the results established in section \ref{7689}, we prove Theorem \ref{kamidaro}, \ref{goddaro}, and \ref{migotoda} in section \ref{9807}. 
In section \ref{5879}, we present some analogous 
results which are not stated in section \ref{9807}.  
In section \ref{65743}, we discuss the connection between the question on products of two Cantor sets 
and the question on intersections of two Cantor sets. 
In section \ref{47325}, we state some open problems.

\section{Preliminaries}\label{765}

\begin{dfn}
For any Cantor set $K \subset \mathbb{R}$, we denote the right and left endpoint of $K$ by $K^{R}$ and $K^{L}$, respectively. 
We denote $K^R - K^L$ by $|K|$. 
If two Cantor sets $K_{1}, K_{2}$ satisfy $K_{1}^{R} < K_{2}^{L}$, we write $K_{1} < K_{2}$.
For any gaps $U$, $U_{1}$ and $U_{2}$, we define $U^{R}$, $U^{L}$, $|U|$ and $U_{1} < U_{2}$ analogously.
For any set $A \subset (0, \infty)$, we denote $\log A$ by $\widetilde{A}$.  
\end{dfn}

\begin{dfn}
We call $K$ an \emph{extended Cantor set} if 
$K$ is a closed, perfect, and nowhere dense set which is 
bounded from above and unbounded from below.
\end{dfn}

The following is immediate:

\begin{lem}
If a Cantor set $K$ satisfies $K > 0$, $\widetilde{K}$ is again a Cantor set. If $K$ is a $0^+$-Cantor set, then  
$\widetilde{K_+}$ is an extended Cantor set.
\end{lem}

\begin{dfn}\label{thickness_def}
Let $K$ be a Cantor set, or an extended Cantor set. Define the \emph{thickness} of $K$ by 
\begin{equation*}
\inf_{U_{1} < U_{2}} \max \left\{ \frac{U^{L}_{2} - U^{R}_{1}}{| U_{1} |} , \frac{U^{L}_{2} - U^{R}_{1}}{| U_{2} |} \right\},
\end{equation*}
where the infimum is taken for all pairs of gaps of $K$, with at least one of them being a finite gap. 
We denote this value by $\tau(K)$.
\end{dfn}

\begin{rem}
Definition \ref{thickness_def} is not the most standard, 
but  if $K$ is a Cantor set it is easy to see that it coincides with the usual definition of thickness. 
Compare the definition in chapter 4 of \cite{PalisTakens}. 
\end{rem}

If we drop the assumption of sizes of $K$ and $L$ in Theorem \ref{sum_sum_sum}, we obtain Theorem \ref{sum_of_Cantor_set}. 
For the reader's convenience, we include the proof (Theorem \ref{sum_sum_sum} can be shown in a similar way).

\begin{thm}\label{sum_of_Cantor_set}
Let $K$ and $L$ be Cantor sets with $\tau(K) \cdot \tau(L) \geq 1$. 
Then, $K + L$ is the disjoint union of finitely many closed intervals.
\end{thm}

\begin{proof}
It is enough to show that $K + L$ has at most finitely many 
open gaps in $[ K^{L} + L^{L}, K^{R} + L^{R} ]$.
Suppose $x \in [ K^{L} + L^{L}, K^{R} + L^{R} ]$ and $x \notin K + L$.
It is easy to see that  
\begin{equation*}
x \notin K + L \iff K \cap ( x - L ) = \phi.
\end{equation*}
Note that $x - L$ is again a Cantor set. Since we have 
\begin{equation*}
\tau(K) \cdot \tau(x -L) = \tau(K) \cdot \tau(L) \geq 1, 
\end{equation*}
the Gap Lemma implies that there are only three possibilities:
\begin{itemize}
\item[(1)] the intervals $[ K^{L}, K^{R} ]$ and $[ x - L^{R}, x - L^{L} ]$ are disjoint; 
\item[(2)] the set $K$ is contained in a finite gap of the set $( x - L )$; 
\item[(3)] the set $(x - L)$ is contained in a finite gap of the set $K$. 
\end{itemize}
Case (1) contradicts the assumption that $x \in [K^{L} + L^{L}, K^{R} + L^{R}]$. 
It is easy to see that the set of points which satisfy (2) or (3) is a union of finitely many 
open intervals. 
\end{proof}

The following is immediate from the definition of thickness:

\begin{lem}\label{straightforward}
Let $K$ be a Cantor set, and let $U$ be a gap of the maximal size of $K$. Let 
$K_{1} = K \cap [K^{L}, U^{L}] \text{ and } \ K_{2} = K \cap [U^{R}, K^{R}]$.
Then  
$\tau(K_{1}) \geq \tau(K) \text{ and \,} \tau(K_{2}) \geq \tau(K)$.
\end{lem}

\begin{lem}\label{1248}
Let $K$ be a Cantor set, and let $C_1, C_2 > 0$ be real numbers. 
Let $f : \mathbb{R} \to \mathbb{R}$ be a strictly increasing continuous function. Assume further that 
for any $p_{1}, p_{2}, p_{3} \in K$ with $p_1 < p_2 < p_3$, we have 
\begin{equation*}
C_{1} \frac{ p_{3} - p_{2} }{ p_{2} - p_{1} }  <  \frac{ f( p_{3} ) - f( p_{2} ) }{ f( p_{2} ) - f( p_{1} ) } < 
C_{2} \frac{ p_{3} - p_{2} }{ p_{2} - p_{1} }.
\end{equation*}
Write $L = f(K)$. Then, we have 
$C_{1} \tau( K ) \leq \tau( L ) \leq C_{2} \tau( K )$.
\end{lem}

\begin{proof}
Let $\epsilon > 0$.
By the definition of thickness there exist two gaps of $L$, say $W_{1}, W_{2}$, 
which satisfy $W_{1} < W_{2}$ and 
\begin{equation*}
\tau( L ) + \epsilon > 
\max \left\{ \frac{ W_{2}^{L} - W_{1}^{R} }{ | W_{1} | },  \frac{ W_{2}^{L} - W_{1}^{R} }{ | W_{2} | }  \right\}.
\end{equation*}
Let $U_{1}, U_{2}$ be the gaps of $K$ with $f(U_{1}) = W_{1}$ and $f(U_{2}) = W_{2}$.
Then, 
\begin{equation*}
\max \left\{ \frac{ W_{2}^{L} - W_{1}^{R} }{ | W_{1} | },  \frac{ W_{2}^{L} - W_{1}^{R} }{ | W_{2} | }  \right\} > 
C_{1} \max \left\{ \frac{ U_{2}^{L} - U_{1}^{R} }{ | U_{1} | },  \frac{ U_{2}^{L} - U_{1}^{R} }{ | U_{2} | }  \right\}
\geq C_{1} \tau(K).
\end{equation*}
Since $\epsilon > 0$ was arbitrary, $\tau(L ) \geq C_{1} \tau(K)$. The other inequality can be shown analogously.
\end{proof}

\begin{cor}\label{so_easy}
Let $\epsilon, c > 0$ be real numbers and assume that $0 < c < 1$. 
Then there exists $\delta > 0$ 
such that for any Cantor set $K$ with $|K| < \delta$ and $K^L > \epsilon$, $\tau( \widetilde{K} )  > c \cdot \tau(K)$.
\end{cor}


\begin{lem}\label{product}
Let $K$ and $L$ be Cantor sets with $\tau(K) \cdot \tau(L) > 1$. Assume that $0 \notin K, L$.
Then $K \cdot L$ is the disjoint union of finitely many closed intervals. 
\end{lem}

\begin{proof}
By Lemma \ref{straightforward}, there exist Cantor sets $K_1, K_2, \cdots, K_{m}$ such that 
\begin{itemize}
\item[(1)] $K = K_1 \sqcup \cdots \sqcup K_{m}$; 
\item[(2)] $\tau(K_i) \geq \tau(K) \ \ (i = 1, 2, \cdots, m)$;
\item[(3)] either $K_i \subset (0, \infty)$ or $K_i \subset (-\infty, 0)$; 
\item[(4)] $|K_i|$ is sufficiently small. 
\end{itemize}
Write $L = L_1 \sqcup \cdots \sqcup L_{n}$ analogously. By Theorem \ref{sum_sum_sum} and Corollary \ref{so_easy},   
$K_i \cdot L_j \ \ (i = 1, 2, \cdots, m, j = 1, 2, \cdots, n)$ is a union of finitely many closed intervals. Therefore, 
$K \cdot L$ is again a union of finitely many closed intervals. The result follows from this. 
\end{proof}

\begin{cor}
Let $K$ and $L$ be Cantor sets with $\tau(K) \cdot \tau(L) > 1$.  
Suppose that $K$ or $L$ contain $0$. Then, one of the following occurs: 
\begin{itemize}
\item[(1)] $(K \cdot L) \cap [0, \infty)$ is the disjoint union of finitely many closed intervals; 
\item[(2)] $( K \cdot L ) \cap [0, \infty)$ is the disjoint union of $\{0\}$ and countably many closed intervals which accumulate to $0$.
\end{itemize}
\end{cor}


\section{Estimates of thickness}\label{7689}

\begin{lem}\label{aaaa}
Let us define a function $f_{k} : \mathbb{R}^+ \to \mathbb{R}^+$ by 
\begin{equation*}
f_{k}(x) = \frac{ \log \left( 1 + \frac{kx}{1 + x} \right) }{ \log (1 + x) },
\end{equation*}
where $k$ is a positive real number. Then $f_{k}$ is a strictly decreasing function. 
In particular, since $\lim_{x \to 0^+} f_k(x) = k$, we have $f_{k}(x) < k$ for all $x \in \mathbb{R}^+$. 
\end{lem}

\begin{proof}
We have 
\begin{equation*}
f_{k}'(x) = \frac{   (1+k)\log( 1 + x ) - \left( 1 + \frac{kx}{1 + x} \right) \log \left( 1 + x + kx \right)  }
{ ( 1 + x + kx ) \left\{ \log (1 + x) \right\}^{2}  }. 
\end{equation*}
Let us denote the numerator by $g_{k}(x)$. Then, since $g_{k}(0) = 0$ and 
\begin{equation*}
g_{k}'(x) = -\frac{ k \log ( 1 + x + kx )  }{ (1+x)^{2} } < 0,
\end{equation*}
we get $g_{k}(x) < 0$. This implies $f_{k}'(x) < 0$.
\end{proof}

\begin{dfn}\label{defdef}
Let $K$ be a Cantor set with $K_{+} \neq \phi$, and let $C$ be a positive number.
If a gap $U$ of $K_{+}$ satisfies
\begin{equation}\label{ske}
\frac{ U^{L} }{ |U| } \geq C,
\end{equation}
we call $U$ a \emph{$C$-nice gap} of $K_+$, 
and a \emph{$C$-bad gap} of $K_+$ otherwise. If the inequality (\ref{ske}) is in fact equality, we call 
$U$ a \emph{$C$-gap} of $K_+$. 
Furthermore, we call $\widetilde{U}$ a \emph{$\log$-$C$-nice gap} of $\widetilde{K_+}$ if $U$ satisfies (\ref{ske}). 
Define a \emph{$\log$-$C$-bad gap} of $\widetilde{K_+}$, and 
a \emph{$\log$-$C$-gap} of $\widetilde{K_+}$ analogously. 
If $U_{1}$, 
$U_{2}$ are $C$-bad gaps of $K_{+}$ which satisfy
\begin{itemize}
\item[(1)] $U_{1} < U_{2}$; 
\item[(2)] every gap contained in $(U^{R}_{1}, U^{L}_{2})$ is a $C$-nice gap of $K_+$; 
\end{itemize}
we call the Cantor set $V = [U^{R}_{1}, U^{L}_{2}] \cap K_{+}$ a \emph{$C$-nice Cantor set} of $K_+$, and 
$\widetilde{V}$ a \emph{$\log$-$C$-nice Cantor set} of $\widetilde{K_+}$.
If
\begin{itemize}
\item[(1)] $K$ is a $0^{+}$-Cantor set; 
\item[(2)] there is no $C$-bad gap of $K_+$ in $(0, K^R)$; 
\end{itemize}
we call $K$ a \emph{$C$-nice $0^{+}$-Cantor set}, 
and $\widetilde{K_+}$ a 
\emph{$\log$-$C$-nice extended Cantor set}.
\end{dfn}

\begin{lem}\label{h}
Let $K$ be a Cantor set with $K_+ \neq \phi$, and let $C$ be a positive number. 
Assume that $U_{1}$ and $U_{2}$ are gaps of $K_{+}$ such that 
\begin{itemize}
\item[(1)] $U_{1} < U_{2}, \  |U_{1}| \leq |U_{2}|$; 
\item[(2)] $U_{1}$  is a  $C$-nice gap of $K_+$.
\end{itemize}
Then, we have 
\begin{equation*}
\frac{\widetilde{U_2}^{L} - \widetilde{U_1}^{R} }{ |\widetilde{U_1}| } \geq 
\frac{  \log \left( 1 + \frac{\tau(K)}{1 + C} \right)  }{ \log \left( 1 + \frac{1}{C} \right) }.
\end{equation*}
\end{lem}

\begin{proof}
Let us write 
\begin{equation*}
x = \frac{|U_{1}|}{U_{1}^{L}}.
\end{equation*}
Note that $x \leq \frac{1}{C}$. By Lemma \ref{aaaa}, we have 
\begin{equation*}
\begin{aligned}
\frac{\widetilde{U_2}^{L} - \widetilde{U_1}^{R} }{ |\widetilde{U_1}| } &\geq  
\frac{  \log \left( U^{L}_{1} + |U_{1}| + \tau(K) |U_{1}|   \right) - 
\log \left( U^{L}_{1} + | U_{1} | \right) }{ \log \left( U^{L}_{1} + |U_{1}| \right)  - \log U^{L}_{1} } \\
& = \frac{  \log \left(  1 + \frac{ \tau(K) x }{ 1 + x }   \right)  }{  \log \left( 1 + x \right)  } \\
&\geq \frac{  \log \left(  1 + \frac{ \tau(K) \frac{1}{C} }{ 1 + \frac{1}{C} }   \right)  }{  \log \left( 1 + \frac{1}{C} \right)  } 
= \frac{ \log \left( 1 + \frac{\tau(K)}{1 + C} \right)  }{ \log \left( 1 + \frac{1}{C} \right) }.
\end{aligned}
\end{equation*}
\end{proof}

\begin{lem}\label{gegege}
Let $K \geq 0$ be a Cantor set, and let $C$ be a positive number. Suppose that every gap of $K$ is 
a $C$-nice gap. Then 
\begin{equation*}
\tau( \widetilde{K_+} ) \geq \frac{ \log \left( 1 + \frac{\tau(K)}{1 + C} \right)  }{ \log \left( 1 + \frac{1}{C} \right) }.
\end{equation*}
\end{lem}

\begin{rem}
If $0 \in K$, then $\widetilde{K_+}$ is an extended Cantor set. 
\end{rem}

\begin{proof}
Let $U_{1}$ and $U_{2}$ be gaps of $K$ with $U_{1} < U_{2}$. \vspace{1mm} \\
Case 1) \ Suppose that $|U_{1}| \leq |U_{2}|$. By Lemma \ref{h}, we have 
\begin{equation*}
\begin{aligned}
\max \left\{ \frac{\widetilde{U_2}^L - \widetilde{U_1}^R}{| \widetilde{U_1} |},  
\frac{\widetilde{U_2}^{L} - \widetilde{U_1}^{R}}{| \widetilde{U_2} |} \right\} 
&\geq \frac{\widetilde{U_2}^L - \widetilde{U_1}^R}{| \widetilde{U_1} |} \\
&\geq \frac{ \log \left( 1 + \frac{\tau(K)}{1 + C} \right)  }{ \log \left( 1 + \frac{1}{C} \right) }.
\end{aligned}
\end{equation*}
Case 2) \ Suppose that $|U_{1}| > |U_{2}|$. By the mean value theorem, we have 
$| \widetilde{U_1} | > | \widetilde{U_2} |$. Therefore, 
\begin{equation*}
\begin{aligned}
\max \left\{ \frac{\widetilde{U_2}^L - \widetilde{U_1}^R}{| \widetilde{U_1} |},  
\frac{\widetilde{U_2}^{L} - \widetilde{U_1}^{R}}{| \widetilde{U_2} |} \right\} 
&= \frac{\widetilde{U_2}^L - \widetilde{U_1}^R}{| \widetilde{U_2} |} \\
&> \tau(K) \\
&> \frac{ \log \left( 1 + \frac{\tau(K)}{1 + C} \right)  }{ \log \left( 1 + \frac{1}{C} \right) }.
\end{aligned}
\end{equation*}
It follows that 
\begin{equation*}
\begin{aligned}
\tau(\widetilde{K_+}) &= \inf_{U_{1}<U_{2}} \max  \left\{ \frac{\widetilde{U_2}^L - \widetilde{U_1}^R}{| \widetilde{U_1} |},  
\frac{\widetilde{U_2}^{L} - \widetilde{U_1}^{R}}{| \widetilde{U_2} |} \right\}  \\
&\geq \frac{ \log \left( 1 + \frac{\tau(K)}{1 + C} \right)  }{ \log \left( 1 + \frac{1}{C} \right) }.
\end{aligned}
\end{equation*}
\end{proof}

\begin{lem}\label{akb}
Let $K$ be a Cantor set with $K_{+} \neq \phi$, and let $C$ be a positive number. 
Assume that $V$ is a $\log$-$C$-nice Cantor set of $\widetilde{K_+}$. Then, 
\begin{equation*}
\begin{aligned}
(1)& \ \ |V| > \log \left( 1 + \frac{\tau(K)}{1 + C} \right); \\
(2)& \ \ \tau( V) \geq \frac{  \log \left( 1 + \frac{\tau(K)}{1 + C} \right)  }{ \log \left( 1 + \frac{1}{C} \right) };  \\
(3)& \ \ \text{for any gap } U \text{ of }  V, \ |U| \leq \log \left( 1 + \frac{1}{C} \right).
\end{aligned}
\end{equation*}
\end{lem}
\begin{proof}
Let $U_{1}, U_{2}$ be the $C$-bad gaps of $K_{+}$ such that 
\begin{equation*}
U_{1} < U_{2}, \ \widetilde{U_1}^R = V^{L}, \ \text{and \ } \widetilde{U_2}^L = V^{R}.
\end{equation*}
If $|U_1| \leq |U_2|$, we get 
\begin{equation*}
\begin{aligned}
|V| = \log \left( 1 + \frac{ U_{2}^{L} - U_{1}^{R} }{ U_{1}^{R} } \right) &>
\log \left(  1 + \frac{ U_{2}^{L} - U_{1}^{R} }{ (1 + C) |U_{1}| }  \right) \\
&\geq \log \left( 1 + \frac{\tau(K)}{1 + C} \right). 
\end{aligned}
\end{equation*}
The other case can be shown similarly. This proves (1). Lemma \ref{gegege} implies (2), and (3) is straightforward. 
\end{proof}

Recall that a Cantor set $K$ is called a $0^{\times}$-Cantor set if $K_{+}, K_{-} \neq \phi$, and $0 \notin K$. 
\begin{dfn}
Let $C$ be a positive number and 
let $K$ be a $0^{\times}$-Cantor set. Assume that $U_{1}, U_{2}$ are the gaps of $K$ which satisfy 
\begin{equation*}
\begin{aligned}
&(1) \ \ 0 \in U_{1}; \\
&(2) \ \ U_{1} < U_{2};  \\
&(3) \ \ \text{every gap in $(U^{R}_{1}, U_{2}^{L})$ is a $C$-nice gap of $K_+$. }
\end{aligned}
\end{equation*}
Then, we call the Cantor set $V = [U^{R}_{1}, U^{L}_{2}] \cap K$ a 
\emph{$C$-nice $0^{\times}$-Cantor set} of $K_+$, and $\widetilde{V}$ a 
\emph{$\log$-$C$-nice $0^{\times}$-Cantor set} of $\widetilde{K_+}$.
\end{dfn}

The following can be shown analogously to Lemma \ref{akb}. 

\begin{lem}\label{819819}
Let $K$ be a $0^\times$-Cantor set, and let $C$ be a positive number. 
Assume that $V$ is the $\log$-$C$-nice $0^{\times}$-Cantor set of $\widetilde{K_+}$.
Then,
\begin{equation*}
\begin{aligned}
(1)& \ \ |V| > \log \left( 1 + \tau(K) \right);  \\
(2)& \ \ \tau(V) \geq \frac{  \log \left( 1 + \frac{\tau(K)}{1 + C} \right)  }{ \log \left( 1 + \frac{1}{C} \right) };  \\
(3)& \ \ \text{for any gap } U \text{ of }  V, \ |U| \leq \log \left( 1 + \frac{1}{C} \right).
\end{aligned}
\end{equation*}
\end{lem}

\begin{lem}\label{extended}
Let $K$ be an extended Cantor set. Then 
for any $\epsilon > 0$, there exist a decreasing sequence $\{ k_{n} \}$ 
which satisfies 
\begin{itemize}
\item[(1)] $k_{n} \to -\infty$; 
\item[(2)] $K \cap [ k_{n}, K^{R}]$ is a Cantor set with 
$\tau( K \cap [ k_{n}, K^{R} ] ) > \tau(K) - \epsilon$. 
\end{itemize}
\end{lem}

\begin{proof}
Without loss of generality, we can assume that $K^R = 0$.
Let 
\begin{equation*}
A = \lim_{N \to -\infty} \sup \left\{ |U| \mid U \text{ is a gap of $K$ contained in } (-\infty, N] \right\}.
\end{equation*}
Case 1) \ Suppose that $A = \infty$. 
We can find a sequence of gaps $\{ U_{n} \} \ ( n = 1, 2, \cdots )$ such that 
\begin{itemize}
\item[(a)] $| U_{n} | > | U |$ \ for every gap $U$ contained in $(U_{n}^{R}, 0)$; 
\item[(b)] $\lvert U_{n+1} \rvert > \lvert U_{n} \rvert$; 
\item[(c)] $U_{n+1} < U_{n}$. 
\end{itemize}
Let $k_{n} = U^{R}_{n} \ (n = 1, 2, \cdots )$. 
By the definition of thickness, it is easy to see that 
$\tau( K \cap [k_{n}, 0] ) \geq \tau( K )$. Therefore, $\{ k_{n} \}$ satisfies the desired properties. \vspace{2mm}\\
Case 2) \ Suppose that $0 < A < \infty$. 
Let us take $\epsilon_{0} > 0$ such that 
\begin{equation*}
\tau(K) \cdot \frac{ A - \epsilon_{0} }{ A + \epsilon_{0} }  > \tau(K) - \epsilon.
\end{equation*} 
Let us choose $N < 0$ which satisfies 
\begin{equation*}
\sup \{ \lvert U \rvert \mid U \text{ is a gap contained in } ( -\infty, N ] \} < A + \epsilon_{0}.
\end{equation*}
Then, we can take a sequence of gaps $\{ U_{n} \} \ ( n = 1, 2, \cdots )$ such that 
\begin{itemize}
\item[(a)] every $U_{n}$ is contained in $( -\infty, N]$; 
\item[(b)] $A - \epsilon_{0} <  \lvert U_{n} \rvert < A + \epsilon_{0}$; 
\item[(c)] $U_{n+1} < U_{n}$.
\end{itemize}
Let $k_{n} = U^{R}_{n} \ ( n = 1, 2, \cdots )$. Let us show that this $\{ k_{n} \}$ satisfies the desired properties.
Retaking $N$ and $\{ U_{n} \}$ if necessary, we can further assume that 
for any gap $S$ contained in $[N, 0]$, $\frac{S^L - U_1^R}{|S|}$ is sufficiently large. 
Therefore, it is enough to show that for any $U_n$ we have 
\begin{equation*}
\frac{ U^{L} - U_{n}^{R}  }{  \lvert U \rvert  } > \tau(K) - \epsilon, 
\end{equation*}
where $U$ is a gap which is contained in $(-\infty, N ]$ and satisfies $U_{n} < U$. 
Let $U$ and $U_{n}$ be such gaps. Then, 
\begin{equation*}
\frac{ U^{L} - U^{R}_{n} }{ \lvert U \rvert } = 
\frac{ U^{L} - U^{R}_{n} }{ \lvert U_{n} \rvert } \cdot 
\frac{ \lvert U_{n} \rvert }{ \lvert U \rvert } \geq \tau(K) \cdot 
\frac{ A - \epsilon_{0} }{ A + \epsilon_{0} } > \tau(K) - \epsilon,   
\end{equation*}
which was to be proved. \vspace{2mm} \\
Case 3) \ Suppose that $A = 0$.
Let us assume that we cannot find such a sequence $\{ k_{n} \}$, to derive a contradiction. 
Let $N < 0$ be a number such that 
\begin{equation*}
\sup \{ \lvert U \rvert \mid U \text{ is a gap contained in } ( -\infty, N ] \} < 1.
\end{equation*}
By assumption, retaking $N$ if necessary, we can assume that 
$\tau( K \cap [U^{R}, 0] ) \leq \tau(K) - \epsilon$ for any gap $U$ contained in $( -\infty, N ]$.
Let us choose a gap $U_{0}$ contained in $( -\infty, N ]$ in such a way that $N - B$ is sufficiently large, where 
\begin{equation*}
B =  U^{R}_{0} + ( 1 + \tau(K) - \epsilon )  \frac{\tau(K)}{\epsilon}.
\end{equation*}

Let us show by induction that we can choose a sequence of gaps 
$U_{n} \ ( n = 1, 2, \cdots )$ contained in $( -\infty, B ]$ which satisfies
\begin{equation}\label{scondition}
\begin{aligned}
(a)& \ \ U_{n} < U_{n+1}; \\
(b)& \ \ \vert U_{n} \rvert \leq \frac{\tau(K) - \epsilon}{\tau(K)} \lvert U_{n+1} \rvert;  \\  
(c)& \ \ U^{L}_{n+1} - U^{R}_{n} \leq ( \tau(K) - \epsilon ) | U_{n+1} |; 
\end{aligned}
\end{equation}  
for all $n = 0, 1, 2, \cdots$. If this is shown, 
this apparently leads to a contradiction since it implies $| U_{n} |$ goes to infinity.

Suppose that $U_{0}, U_{1}, \cdots , U_{n}$ are already chosen.
Since
\begin{equation*}
\tau( K \cap [ U^{R}_{n}, 0 ] ) \leq \tau(K) - \epsilon
\end{equation*}
and $N - B$ is sufficiently large, 
there exists a gap $U_{n+1}$ contained in $( -\infty, N ]$ which satisfies
\begin{equation*} 
\begin{aligned}
(a)& \ \ U_{n} < U_{n+1}; \\ 
(b)& \ \ \frac{ U^{L}_{n+1} - U^{R}_{n} }{ | U_{n+1} | } \leq  \tau(K) - \epsilon;  \\
(c)& \ \ \frac{ U^{L}_{n+1} - U^{R}_{n} }{ | U_{n} | } \geq \tau(K).
\end{aligned}
\end{equation*}
Therefore,
\begin{equation*}
\vert U_{n} \rvert \leq \frac{\tau(K) - \epsilon}{\tau(K)} \lvert U_{n+1} \rvert, 
\text{ \ and  \ }  U^{L}_{n+1} - U^{R}_{n} \leq ( \tau(K) - \epsilon ) | U_{n+1} |.
\end{equation*} 
Since we have 
\begin{equation*}
| U_{i+1} | + ( U_{i+1}^{L} - U_{i}^{R} ) \leq ( 1 + \tau(K) -\epsilon ) | U_{i+1} | \ \ ( i = n, n-1, \cdots , 0 ), 
\end{equation*}
we get
\begin{equation*}
\begin{aligned}
U^{R}_{n+1} &= U^{R}_{0} + \left\{  \lvert U_{n+1} \rvert  + ( U^{L}_{n+1} - U^{R}_{n} ) +  
\lvert U_{n} \rvert + ( U^{L}_{n} - U^{R}_{n-1} ) + \cdots + 
\lvert U_{1} \rvert + ( U^{L}_{1} - U^{R}_{0} ) \right\}  \\
&< U^{R}_{0}  +  \left\{ ( 1 + \tau(K) - \epsilon ) | U_{n+1} | + ( 1 + \tau(K) -  \epsilon ) | U_{n} | + \cdots 
\right\} \\
&< U_{0}^{R} + ( 1 + \tau(K) - \epsilon ) \left\{  1 + \frac{\tau(K) - \epsilon }{ \tau(K) } 
+ \left(  \frac{\tau(K)-\epsilon}{\tau(K)} \right)^{2} + \cdots  
\right\} \\
&= B
\end{aligned}
\end{equation*}
(from the second to the third inequality, we used the fact that $|U_{n+1}| < 1$).
Therefore, $U_{n+1}$ is indeed contained in $( -\infty, B]$.

By induction, we can choose a sequence of gaps 
$\{ U_{n} \}$ which satisfies the condition (\ref{scondition}), a clear contradiction.
\end{proof}

\begin{lem}\label{nmb}
Let $C > 0$ be a real number, and 
let $K$ be a $C$-nice $0^{+}$-Cantor set. 
Then, there exists a decreasing sequence $\{ k_{n} \}$ 
which satisfies 
\begin{itemize}
\item[(1)] $k_{n} \to -\infty$; 
\item[(2)] $\widetilde{K_+} \cap [ k_{n}, \widetilde{K_+}^{R} ]$ is a Cantor set with 
\begin{equation*}
\tau \left( \widetilde{K_+} \cap [ k_{n}, \widetilde{K_+}^{R} ] \right) \geq 
\frac{ \log \left( 1 + \frac{\tau(K)}{1 + C} \right) }{ \log \left( 1 + \frac{1}{C} \right) }. 
\end{equation*}
\end{itemize}
\end{lem}

\begin{proof}
Let 
\begin{equation*}
A = \lim_{N \to -\infty} \sup \left\{ |U| \mid U \text{ is a gap of $\widetilde{K_+}$ contained in } (-\infty, N] \right\}.
\end{equation*}
Note that $A \leq \log( 1 + \frac{1}{C})$.
\vspace{1mm} \\
Case 1) \ Suppose that $A = \log \left( 1 + \frac{1}{C} \right)$. 
Then, there exists a sequence of gaps $\{ U_{n} \} \ (n = 0, 1, 2, \cdots )$ of $\widetilde{K_+}$ such that 
\begin{equation*}
\begin{aligned}
(a)& \ \  | U_{n} | \uparrow \log \left( 1 + \frac{1}{C} \right);  \\
(b)& \ \ U_{0} > U_{1} > U_{2} > \cdots;  \\
(c)& \ \ | U_{n} | \geq |U| \text{ \ for every gap } U \text{ in } [U^{R}_{n}, \widetilde{K_+}^{R}].  
\end{aligned}
\end{equation*}
By Lemma \ref{gegege}, 
$k_{n} = U^{R}_{n} \ (n = 0, 1, 2, \cdots)$ satisfies the desired properties. \vspace{1mm} \\
Case 2) \ Suppose that every gap in $K$ is a $C'$-nice gap for some $C' > C$. Then, since we have 
\begin{equation*}
\tau( \widetilde{K_+} ) \geq \frac{ \log \left( 1 + \frac{\tau(K)}{1 + C'} \right) }{ \log \left( 1 + \frac{1}{C'} \right) }
> \frac{ \log \left( 1 + \frac{\tau(K)}{1 + C} \right) }{ \log \left( 1 + \frac{1}{C} \right) }
\end{equation*} 
by Lemma \ref{gegege} and Lemma \ref{aaaa}, the claim follows from Lemma \ref{extended}. \vspace{1mm} \\
Case 3) \ 
Suppose that $A < \log \left( 1 + \frac{1}{C} \right)$, and $K$ has a $C$-gap. Let $U$ be the $\log$-$C$-gap of 
$\widetilde{K_+}$ such that 
each gap in $(-\infty, U^{L}]$ is not a $\log$-$C$-gap of $\widetilde{K_+}$. Then, there exists $C'' > C$ such that 
every gap of $\widetilde{K_+}$ in $(-\infty, U^{L}]$ is a $\log$-$C''$-nice gap of $\widetilde{K_+}$. 
Therefore, we get 
\begin{equation*}
\tau \left( (-\infty, U^{L}] \cap \widetilde{K_+} \right) \geq \frac{ \log \left( 1 + \frac{\tau(K)}{1 + C''} \right) }{ \log \left( 1 + \frac{1}{C''} \right) } > 
\frac{ \log \left( 1 + \frac{\tau(K)}{1 + C} \right) }{ \log \left( 1 + \frac{1}{C} \right) },
\end{equation*}
so again the claim follows from Lemma \ref{extended}.
\end{proof}

\section{Proof of the main results}\label{9807}
In this section, we will prove our main results. 

\begin{proof}[Proof of Theorem \ref{kamidaro}]
Without loss of generality, we can assume that 
\begin{equation*}
\tau(L) \geq  \frac{2 \tau(K) + 1 }{ \tau(K)^2 }.
\end{equation*}
Let $C = \tau(K) \tau(L) - 1$. Note that $C > 0$. It is easy to see that 
\begin{equation*}
1 + \frac{\tau(L)}{1 + C} = 1 + \frac{1}{\tau(K)}, \ \text{ and } \  1 + \frac{1}{C} \leq 1 + \frac{\tau(K)}{1 + \tau(K)}.
\end{equation*}
Case 1) \ Suppose that every gap in $L$ is a $C$-nice gap. 
Note that every gap in $K$ is a $\tau(K)$-nice gap.  
Since  
\begin{equation*}
\frac{ \log \left( 1 + \frac{\tau(K)}{1 + \tau(K)} \right) }{ \log \left( 1 + \frac{1}{\tau(K)} \right) } \cdot
 \frac{ \log \left(  1 + \frac{\tau(L)}{ 1 + C }  \right)  }{ \log \left( 1 + \frac{1}{C} \right) } \geq 1, 
\end{equation*}
By Theorem \ref{sum_sum_sum} and Lemma \ref{nmb}, $\widetilde{K_+} + \widetilde{L_+}$ is a half line. 
\vspace{1mm} \\
Case 2) \ Suppose that there is a $C$-bad gap in $L$. 
Let $V$ be the $C$-nice Cantor set of $L$ such that $V^R = L^R$. By Lemma \ref{akb}, 
\begin{equation*}
\begin{aligned}
| \widetilde{V} | &> \log \left(  1 + \frac{\tau(L)}{ 1 + C } \right) \\
&= \log \left( 1 + \frac{1}{\tau(K)} \right), 
\end{aligned}
\end{equation*}
so for any gap $U$ of $\widetilde{K_+}$, we have $|U| \leq |\widetilde{V}|$. 
Therefore, Theorem \ref{sum_sum_sum} and 
Lemma \ref{nmb} imply that $\widetilde{K_+} + \widetilde{L_+}$ is a 
half line.

Next, let $M, N > 0$ be real numbers with condition (\ref{thm0}). 
Without loss of generality, we can assume that $M \geq N$. With this additional assumption, (\ref{thm0}) is equivalent to 
\begin{equation*}
N < \frac{2 M + 1}{M^2}.
\end{equation*}
Let 
$C = \frac{ M ( 1 + N ) }{ 1 + M }$.
Then, it is easy to see that 
\begin{equation*}
1 + \frac{1}{M} > 1 + \frac{N}{1 + C},  \ \ \frac{1 + 2M}{M} = \frac{1 + C + N}{C}, \text{ \ and } \ \  C \geq N.
\end{equation*}
Let $K$ be the middle-$\frac{1}{1 + 2M}$ Cantor set whose convex hull is $[0, 1]$. 
Write 
\begin{equation*}
K_{0} = K \cap \left[\frac{1+M}{1+2M}, 1 \right].
\end{equation*}
Then 
\begin{equation*}
K = \{0\} \sqcup \bigsqcup_{n=0}^{\infty} \left( \frac{M}{ 1 + 2M } \right)^{n} K_{0}.
\end{equation*}
Let $L_{0}$ be a Cantor set with sufficiently large thickness whose convex hull is 
$\left[\frac{1+ C}{ 1 + C + N}, 1\right]$. Let us define a $0^{+}$-Cantor set $L$ as follows:
\begin{equation*}
L  = \{0\} \sqcup \bigsqcup_{n=0}^{\infty} \left( \frac{C}{ 1 + C + N } \right)^{n} L_{0}.
\end{equation*}
Since $\tau(L_0)$ is sufficiently large, $\tau(L) = N$. Note that 
\begin{equation*}
\begin{aligned}
\widetilde{K_+} + \widetilde{L_+} &= 
\bigcup_{m, n \geq 0} \left\{  -m \log \left( \frac{1 + 2M}{M} \right) - n \log \left( \frac{1 + C + N}{C} \right) 
+ \left( \widetilde{K_0} + \widetilde{L_0} \right)  \right\} \\
&= \bigcup_{n \geq 0} \left\{  -n \log \left( \frac{1 + 2M}{M} \right) 
+ \left( \widetilde{K_0} + \widetilde{L_0} \right)  \right\}.
\end{aligned}
\end{equation*}
Note that by Theorem \ref{sum_sum_sum} and Lemma \ref{gegege}, $\widetilde{K_0} + \widetilde{L}_0$ is an interval. 
Therefore, since 
\begin{equation*}
\begin{aligned}
( \widetilde{K_0}^R + \widetilde{L_0}^R ) - ( \widetilde{K_0}^L + \widetilde{L_0}^L ) &= 
( \widetilde{K_0}^R - \widetilde{K_0}^L ) + ( \widetilde{L_0}^R - \widetilde{L_0}^L ) \\
&= \log \left( 1 + \frac{M}{1 + M} \right) + \log \left( 1 + \frac{N}{1 + C} \right) \\
&< \log \left( 1 + \frac{M}{1 + M} \right) + \log \left( 1 + \frac{1}{M} \right) \\
&= \log \left( \frac{1 + 2M}{M} \right),
\end{aligned}
\end{equation*}
$\widetilde{K_+} + \widetilde{L_+}$ is a disjoint union of countably many closed intervals. 

Next, let us also show that $K \cdot L$ can be a disjoint union of $k$ intervals, for any $k \geq 2$. 
The construction is almost the same as the construction above. 
Consider exactly the same $K$ and $L_0$, and modify $L$ to be 
\begin{equation*}
L = L_{1} \sqcup \bigsqcup_{n=0}^{k-2} \left( \frac{C}{1 + C + N} \right)^{n} L_{0},
\end{equation*}
where $L_{1}$ is a Cantor set which satisfies 
\begin{equation*}
\begin{aligned}
(1)& \ \ \text{the convex hull of $L_{1}$ is} \left[ 0, \left( \frac{C}{ 1 + C + N } \right)^{k-1} \right];  \\
(2)& \ \ \tau(L_{1}) \text{ is sufficiently large.}
\end{aligned}
\end{equation*}
Then, $\tau(K) = M, \, \tau(L) = N$, and it is easy to see that $K \cdot L$ is a disjoint union of $k$ intervals. 
\end{proof}

\begin{rem}
Since every gap in $K$ and $L$ is a $\tau(K)$-nice gap and a $\tau(L)$-nice gap, respectively, 
$K \cdot L$ is an interval if 
\begin{equation*}
\frac{ \log \left( 1 + \frac{\tau(K)}{1 + \tau(K)} \right) }{ \log \left( 1 + \frac{1}{\tau(K)} \right) } \cdot
\frac{ \log \left( 1 + \frac{\tau(L)}{1 + \tau(L)} \right) }{ \log \left( 1 + \frac{1}{\tau(L)} \right) } \geq 1, 
\end{equation*}
 by Lemma \ref{nmb}. 
But interestingly this is not the optimal estimate.
\end{rem}

Next, let us show Theorem \ref{goddaro}. The proof is essentially the same as Theorem \ref{kamidaro}. 

\begin{proof}[proof of Theorem \ref{goddaro}]
The first half is a verbatim repetition of Theorem \ref{kamidaro}. To show that the estimate is optimal, 
we consider 
$L'$ instead of $L$, which has the same positive part as $L$ and has the negative part of sufficiently large size and thickness.  
\end{proof}

To prove Theorem \ref{migotoda}, we need a series of lemmas. 
Recall that for any Cantor set $K$ and real number $C>0$, we defined $\log$-$C$-nice Cantor set of $K_{+}$, and 
$\log$-$C$-nice extended Cantor set of $K_{+}$ in Definition \ref{defdef}. We define 
$\log$-$C$-nice Cantor set of $K_{-}$, and $\log$-$C$-nice extended Cantor set of $K_{-}$ analogously. 

\begin{lem}\label{819}
Let $K$ be a Cantor set, and let $C < \tau(K)$ be a positive number. 
Suppose that $U$ is a $\log$-$C$-bad gap of $\widetilde{K_+}$. 
Then, there exists a set $X \subset \widetilde{K_-}$ such that
\begin{equation*}
\begin{aligned}
&(1) \ X \text{ is a } \log\text{-}C\text{-nice Cantor set of $\widetilde{K_-}$, or a } \log \text{-} C\text{-nice extended Cantor set of } 
\widetilde{K_-}; \\
&(2) \ X^{R} - U^{R} > \log \left( \frac{ \tau(K) - C }{ 1 + C } \right);  \\
&(3) \ U^{L} - X^{L} > \log \left( \frac{ \tau(K) - C }{ 1 + C } \right)
\end{aligned}
\end{equation*}
(if $X$ is a $\log$-$C$-nice extended Cantor set, we set $X^{L} = -\infty$). 
\begin{rem}
In general, this set $X$ is not unique. 
\end{rem}
\end{lem}

\begin{proof}
Let $W$ be the gap of $K_{+}$ satisfying $\widetilde{W} = U$.
Let $U_{1}$ be the $C$-bad gap of $K_-$ which satisfies 
\begin{itemize}
\item[(1)] $|W| \leq |U_{1}|$; 
\item[(2)] the size of every $C$-bad gap of $K_-$ in $(0, U_1^{L})$ is less than $|W|$. 
\end{itemize}
Then, by the definition of thickness, we have 
\begin{equation*}
U_{1}^{L} > \tau(K) |W| - C |W|.
\end{equation*}
Therefore, 
\begin{equation*}
\begin{aligned}
\widetilde{U_1}^{L} - U^{R} &= \log U_{1}^{L} - \log W^{R} \\
&> \log ( \tau(K) - C ) |W| - \log ( 1 + C ) |W| \\
&= \log \left( \frac{ \tau(K) - C }{ 1 + C } \right).
\end{aligned}
\end{equation*}
Suppose next that $U_{2}$ is the $C$-bad gap of $K_{-}$ such that 
\begin{itemize}
\item[(1)] $|U_{2}| < |W|$; 
\item[(2)] every gap in $(U^{R}_2, U^{L}_1) \cap K_{-}$ is a $C$-nice gap. 
(If such $U_2$ does not exist, then $X = \log\left( {(0, U^L_1] \cap K_-} \right)$ is a desired set.)
\end{itemize}
Then, by the definition of thickness, we have 
\begin{equation*}
W^{L} > \tau(K) |U_2| - C |U_2|.
\end{equation*}
Therefore, 
\begin{equation*}
\begin{aligned}
U^{L} - \widetilde{U_2}^{R} &= \log W^{L} - \log U_2^{R} \\
&> \log ( \tau(K) - C ) |U_2| - \log ( 1 + C ) |U_2| \\
&= \log \left( \frac{ \tau(K) - C }{ 1 + C } \right).
\end{aligned}
\end{equation*}
Then $X = \log \left( K_{-} \cap [U_2^R, U_1^L] \right)$ satisfies the desired properties.  
\end{proof}

We call the set $X$ given in Lemma \ref{819} a $\log$-$C$-\emph{cover} of $U$.

\begin{lem}\label{what}
Let $x, y$ be positive real numbers with $2(x + 1)( y + 1 ) \leq ( xy - 1 )^{2}$. 
Write
\begin{equation*}
C_{x, y} = \frac{x + 1}{xy - 1}, \text{ and } \ C_{y, x} = \frac{y + 1}{xy - 1}.
\end{equation*}
Then, $0 < C_{x, y} < x$ and $0 < C_{y, x} < y$. 
Furthermore, we have 
\begin{equation*}
\frac{x - C_{x, y}}{1 + C_{x, y}} \cdot \frac{y - C_{y, x}}{1 + C_{y, x}} \geq 1.
\end{equation*}
\end{lem}

\begin{proof}
It is easy to see that 
\begin{equation*}
2(x + 1)( y + 1 ) \leq ( xy - 1 )^{2} \Longleftrightarrow y 
\geq \frac{ 2x + 1 + ( x + 1 ) \sqrt{ 2x + 1 } }{ x^{2} }.
\end{equation*}
This implies that $xy > 1$. Therefore, since 
\begin{equation*}
 \frac{x + 1}{xy - 1} < x \Longleftrightarrow \frac{2x + 1}{x^{2}} < y, 
\end{equation*} 
the first claim follows. Also, since we have 
\begin{equation*}
\begin{aligned}
\frac{x - C_{x, y}}{1 + C_{x, y}} \cdot \frac{y - C_{y, x}}{1 + C_{y, x}} \geq 1 
& \Longleftrightarrow xy -1 \geq  C_{y, x} ( x + 1 ) +  C_{x, y} ( y + 1 )  \\
&\Longleftrightarrow xy - 1 \geq \frac{ (x+1)(y+1) }{ xy - 1 } + \frac{(x+1)(y+1)}{xy - 1}  \\
&\Longleftrightarrow ( xy - 1 )^{2} \geq 2 (x+1)(y+1),
\end{aligned}
\end{equation*}
the second claim follows.
\end{proof}

For any Cantor sets $K$ and $L$ with $\tau(K) \cdot \tau(L) > 1$, define $C_{K, L}$ and $C_{L, K}$ by  
\begin{equation*}
C_{K, L} = \frac{\tau(K) + 1}{ \tau(K) \tau(L) -1 }, \text{ and } \ C_{L, K} = \frac{ \tau(L) + 1 }{ \tau(K) \tau(L) - 1 }.
\end{equation*}

\begin{lem}\label{123}
Let $K$ and $L$ be Cantor sets with $\tau(K) \cdot \tau(L) > 1$. Assume that $V$ is a 
$\log$-$C_{K, L}$-nice Cantor set of $\widetilde{K_+}$, 
or a $\log$-$C_{K, L}$-nice $0^{\times}$-Cantor set of $\widetilde{K_+}$, or a $\log$-$C_{K, L}$-nice extended Cantor set of 
$\widetilde{K_+}$. 
Similarly, suppose that $T$ is a 
$\log$-$C_{L, K}$-nice Cantor set of $\widetilde{L_+}$, or a $\log$-$C_{L, K}$-nice $0^{\times}$-Cantor set of $\widetilde{L_+}$, 
or a $\log$-$C_{L, K}$-nice extended Cantor set of $\widetilde{L_+}$. 
Then $V + T$ is an interval, or a half line. 
\end{lem}

\begin{proof}
Let us consider the case where $V$ is a $\log$-$C_{K, L}$-nice Cantor set of $\widetilde{K_+}$, and $T$ is a 
$\log$-$C_{L, K}$-nice Cantor set of $\widetilde{L_+}$. 
Since 
\begin{equation*}
1 + \frac{\tau(K)}{1 + C_{K, L}} = 
1 + \frac{1}{C_{L, K}},  \text{ and } \
1 + \frac{\tau(L)}{1 + C_{L, K}} = 
1 + \frac{1}{C_{K, L}}, 
\end{equation*}
by Theorem \ref{sum_sum_sum} and Lemma \ref{akb}, $V + T$ is an interval. Other cases can be shown analogously. 
\end{proof}

\begin{lem}\label{128}
Let $K, L$ be Cantor sets with condition (\ref{46578}). 
Let $U_{1}$, $U_{2}$ be a $\log$-$C_{K, L}$-bad gap of $\widetilde{K_+}$, and a 
$\log$-$C_{L, K}$-bad gap of $\widetilde{L_+}$, respectively. 
Let $X, Y$ be a $\log$-$C_{K, L}$-cover of $U_{1}$, and a $\log$-$C_{L, K}$-cover of $U_{2}$, respectively. 
Then, we have $X + Y  \supset U_{1} + U_{2}$. 
\end{lem}

\begin{proof}
By Lemma \ref{what}, $\tau(K) > C_{K, L}, \ \tau(L) > C_{L, K}$, 
and 
\begin{equation*}
\frac{ \tau(K) - C_{K, L} }{ 1 + C_{K, L} } \cdot \frac{ \tau(L) - C_{L, K} }{ 1 + C_{L, K} } \geq 1.
\end{equation*}
Therefore, by Lemma \ref{819} we have 
\begin{equation*}
\begin{aligned}
( X^{R} + Y^{R} ) - ( U^{R}_{1} + U^{R}_{2} ) 
&= ( X^{R} - U^{R}_{1} ) + ( Y^{R} - U^{R}_{2} ) \\
&\geq \log \frac{ \tau(K) - C_{K, L} }{ 1 + C_{K, L} } +  \log \frac{ \tau(L) - C_{L, K} }{ 1 + C_{L, K} } \\
&\geq 0.
\end{aligned}
\end{equation*}
Similarly, $( U^{L}_{1} + U^{L}_{2} ) - ( X^{L} + Y^{L} ) \geq 0$. 
The claim follows from this and Lemma \ref{123}.
\end{proof}

\begin{dfn}
Let $K$ be a Cantor set with $K_{+} \neq \phi$, and let $C$ be a positive number. 
Let us take the sequence of $\log$-$C$-bad gaps $\{ U_{n} \} \ ( n = 0, 1, \cdots, k )$ of $\widetilde{K_+}$, where $k$ is 
either finite or infinite, in the following way:
\begin{itemize}
\item[(1)] $U_{0} = ( \widetilde{K_+}^R, \infty)$; 
\item[(2)] $U_{0} > U_{1} > U_{2} > U_{3} > \cdots$; 
\item[(3)] $V_{n} = [U_{n+1}^{R}, U_{n}^{L}] \cap \widetilde{K_+} \ (n = 0, 1, 2, \cdots )$ 
are $\log$-$C$-nice Cantor sets of $\widetilde{K_+}$. 
(If $k$ is finite, set $V_k = (-\infty, U^{L}_{k}]$.)
\end{itemize}
Then, we call $\{ U_{n} \}$ and $\left\{ V_{n} \right\}$ the \emph{$\log$-$C$-split gaps} of $\widetilde{K_+}$ and 
\emph{$\log$-$C$-split Cantor sets} of $\widetilde{K_+}$, respectively. If $k$ is finite, we say this split is \emph{finite}.
\end{dfn}

\begin{rem}
Note that if the split is finite in the above definition, $V_{k}$ is either a $\log$-$C$-nice extended Cantor set, or 
a $\log$-$C$-nice $0^{\times}$-Cantor set.
\end{rem}

Using these lemmas, we can complete the proof of Theorem \ref{migotoda}.

\begin{proof}[proof of Theorem \ref{migotoda}]
Let $\{ U_{n} \} \ (n = 0, 1, \cdots, k)$ and $\{ V_{n} \} \ (n = 0, 1, \cdots, k)$ be the $\log$-$C_{K, L}$-split gaps 
of $\widetilde{K_+}$, and the 
$\log$-$C_{K, L}$-split Cantor sets of $\widetilde{K_+}$, respectively.
Similarly, let $\{ S_{n} \} \ (n = 0, 1, \cdots, l) $ and $\{ T_{n} \} \ (n = 0, 1, \cdots, l)$ 
be the $\log$-$C_{L, K}$-split gaps of $\widetilde{L_+}$, and the $\log$-$C_{L, K}$-split Cantor sets of $\widetilde{L_+}$, respectively.
Then, by Lemma \ref{123},
\begin{equation*}
V_{i} + T_{j} \ ( i = 0, 1, \cdots, k, \  j = 0, 1, \cdots, l)
\end{equation*} 
are intervals, or half lines.
Let $X_{n} \ (n = 1, 2, \cdots, k)$ and $Y_{n} \ (n = 1, 2, \cdots, l)$ be $\log$-$C_{K, L}$-cover of $U_{n}$, and 
$\log$-$C_{L, K}$-cover of $S_{n}$, respectively. Then, by Lemma \ref{123} and Lemma \ref{nmb}, 
\begin{equation*}
X_{i} + Y_{j} \ (i = 1, 2, \cdots, k, \ j= 1, 2, \cdots, l)
\end{equation*} 
are intervals, or half lines. Note that, by Lemma \ref{128}, we have 
\begin{equation}
X_{i} + Y_{j} \supset U_{i} + S_{j} \  ( i = 1, 2, \cdots, k, \  j = 1, 2, \cdots, l).
\end{equation}
Consider the case that $k$ and $l$ are both infinite. Other cases can be shown similarly. We get 
\begin{equation*}
\begin{aligned}
( \widetilde{K_+} + \widetilde{L_+} ) \cup ( \widetilde{K_-} + \widetilde{L_-} ) 
&\supset \bigcup_{i=0}^{\infty} ( V_{i} + T_{i} )  \cup 
\bigcup_{i=1}^{\infty} ( X_{i} + Y_{i} ) \\
&\supset (-\infty, \widetilde{K_+}^R + \widetilde{L_+}^R].
\end{aligned}
\end{equation*}
Similarly, 
$( \widetilde{K_+} + \widetilde{L_+} ) \cup 
( \widetilde{K_-} + \widetilde{L_-} ) \supset (-\infty, \widetilde{K_-}^R + \widetilde{L_-}^R]$.
Therefore, 
\begin{equation*}
( \widetilde{K_+} + \widetilde{L_+} ) \cup 
( \widetilde{K_-} + \widetilde{L_-} ) = 
\left( -\infty, \max \left\{ \widetilde{K_+}^R + \widetilde{L_+}^R, \widetilde{K_-}^R + \widetilde{L_-}^R \right\} \right], 
\end{equation*}
which implies the first claim of the Theorem.

Next, let us show that this estimate is optimal.
Let $M, N$ be positive real numbers with condition (\ref{thm3}).
Let 
\begin{equation*}
C_1 = \frac{ MN(M+1) }{3MN + 2M + 2N + 1}, \text{ and } \ C_2 = \frac{ MN(N+1) }{3MN + 2M + 2N + 1}.
\end{equation*}
Note that (\ref{thm3}) implies $C_1 < M$ and $C_2 < N$. 
It is easy to see that 
\begin{equation}\label{oi}
\frac{ 1 + C_1 + M }{ M - C_1 } \cdot \frac{ 1 + C_2 + N }{N - C_2} 
= 1 + \frac{ 1 + M }{C_1} = 1 + \frac{1 + N}{C_2}. 
\end{equation}
Also, (\ref{thm3}) implies 
\begin{equation}\label{io}
1 + \frac{1}{C_1} > 1 + \frac{N}{1 + C_2}, \ \ 1 + \frac{1}{C_2} > 1 + \frac{M}{1 + C_1}. 
\end{equation}
Let $K$ be a $0$-Cantor set such that 
\begin{itemize}
\item[(1)] $K = K_{1} \sqcup K_{2}$ and $0 \in K_{1}$; 
\item[(2)] $K_{1}, K_{2}$ are Cantor sets with sufficiently large thickness; 
\item[(3)] the convex hull of $K_{1}$ and $K_{2}$ are $[C_{1} - M, C_{1}]$ and $[ 1 + C_{1}, 1 + C_{1} + M]$, respectively.
\end{itemize}
Let us define a $0$-Cantor set $L = L_1 \sqcup L_2$ analogously, with $C_2$ instead of $C_1$ and $N$ instead of $M$. 
Note that $\tau(K) = M$ and $\tau(L) = N$. 
Let $U = \left( K^{R}_{1}, K^{L}_{2} \right)$ and $S = \left( L^{R}_{1}, L^{L}_{2} \right)$. 
By (\ref{oi}), we have 
\begin{equation*}
|\widetilde{K_2}| + |\widetilde{U}| = |\widetilde{L_2}| + |\widetilde{S}|,  \text{ and  }  \
\widetilde{K_-}^{R} + \widetilde{L_-}^{R} = \widetilde{U}^{L} + \widetilde{S}^{L}. 
\end{equation*}
Also, (\ref{io}) implies 
\begin{equation*}
|\widetilde{K_2}| < |\widetilde{S}|,  \text{ and } \ |\widetilde{L_2}| < |\widetilde{U}|. 
\end{equation*}
Therefore, 
\begin{equation*}
( \widetilde{K_+} + \widetilde{L_+} ) \cup ( \widetilde{K_-} + \widetilde{L_-} ) 
= (-\infty, \widetilde{U}^L + \widetilde{S}^L] \sqcup [\widetilde{K_2} + \widetilde{L_2}]. 
\end{equation*}
\end{proof}

Next, let us show Theorem \ref{43215}. We need the following definition and lemma:

\begin{dfn}
Let $K$ be a $0$-Cantor set, and let $C, M$ be positive numbers. If 
\begin{equation*}
\begin{aligned}
(1)& \ \ \text{$K_{0}$ is a Cantor set whose convex hull is } \left[ \frac{1 + C}{1 + C + M}, 1 \right]; \\
(2)& \ \ \tau(K_{0}) > \max \left\{ 1, \, M,  \, C + \sqrt{ C (1 + C + M) } \right\}; \\
(3)& \ \ K_{+} = \bigsqcup_{n=0}^{\infty} \left( \frac{C}{1 + C + M} \right)^{n} K_{0}; \\
(4)& \ \ K_{-} = \sqrt{ \frac{C}{ 1 + C + M } } \, K_{+};
\end{aligned}
\end{equation*}
we call $K$ a \emph{$(C, M)$-Cantor set}.
Note that 
\begin{equation*}
\widetilde{K_+} = \bigsqcup_{n=0}^{\infty} \left( \widetilde{K_0} - nd \right),  \text{ and } \ 
\widetilde{K_-} = \bigsqcup_{n=0}^{\infty} \left( \widetilde{K_0} - \left( n + \frac{1}{2} \right) d  \right), 
\end{equation*}
where 
\begin{equation*}
d =  \log \left( 1 + \frac{1 + M}{C} \right). 
\end{equation*}
See Figure \ref{hihihihihi}.
\end{dfn}

\begin{centering}
\begin{figure}[t]
\includegraphics[scale=1.00]{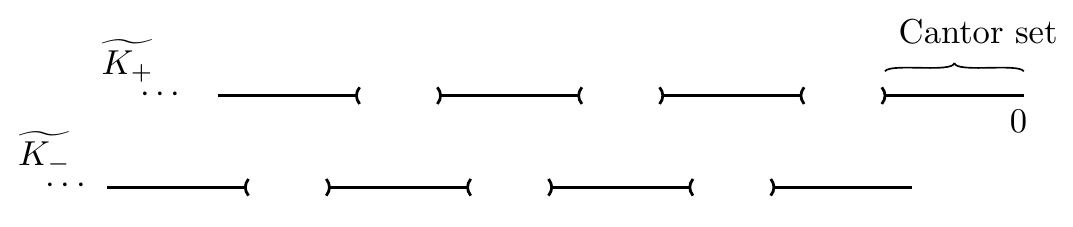}
\caption{$K$ is a $(C, M)$-Cantor set (line segments represent Cantor sets)}
\label{hihihihihi}
\end{figure}
\end{centering}

\begin{lem}\label{chain}
Let $C, M > 0$ be real numbers and let 
$K$ be a $(C, M)$-Cantor set. We have  
\begin{equation*}
\begin{aligned}
(1)& \ \ \text{if } \ C \geq \frac{M^{2}}{3M + 1}, \text{ \ then \ } \tau(K) = M;   \\
(2)& \ \ \text{if } \ C \leq \frac{M^2}{3M + 1},  \text{ \ then \ }  \tau(K) = C + \sqrt{ C (1 + C + M) }.
\end{aligned}
\end{equation*}
\end{lem}

\begin{proof}
Write $U = \left( \frac{C}{1 + C + M}, \frac{1 + C}{1 + C + M} \right)$. 
By the definition of $(C, M)$-Cantor set, we have 
\begin{equation*}
\begin{aligned}
\tau(K) &= \min \left\{ \frac{K^R - U^R}{|U|}, \frac{U^L - K^L}{|U|} \right\} \\
&= \left\{ M, \, C + \sqrt{C( 1 + C + M )} \right\}.  
\end{aligned}
\end{equation*}
By a simple computation, we get 
\begin{equation*}
C + \sqrt{C( 1 + C + M )} \geq M \Longleftrightarrow C \geq \frac{M^{2}}{ 3M + 1 }.
\end{equation*}
The result follows from this. 
\end{proof}

\begin{proof}[proof of theorem \ref{43215}]
First, let us assume that 
\begin{equation}\label{ehe}
M \geq N, \text{ and } \ N < \frac{(2M+1)^2}{M^3}. 
\end{equation}
Let 
\begin{equation*}
C_{1} = \frac{M^2}{1 + 3M}, \text{ and } \ C_{2} = \frac{M^2}{(1 + M)( 1 + 3M)} (1 + N).
\end{equation*}
Then, we have 
\begin{equation*}
1 + \frac{1 + M}{C_{1}} = 1 + \frac{1 + N}{C_{2}}. 
\end{equation*}
Also, it is easy to see that (\ref{ehe}) implies 
\begin{equation*}
C_{2} \geq \frac{N^2}{3 N + 1}, \text{ and } \ 1 + \frac{1}{C_{1}} > 1 + \frac{N}{1 + C_{2}}. 
\end{equation*}
Let $K$, $L$ be a $(C_{1}, M)$-Cantor set and a $(C_{2}, N)$-Cantor set, respectively.  
By Lemma \ref{chain} we have $\tau(K) = M \text{ and } \tau(L) = N$. 
It is easy to see that $K \cdot L$ is a disjoint union of $\{0\}$ and countably many closed intervals. 

Next, let us assume that 
\begin{equation*}
M \geq N, \text{ and } \ M < \frac{N^2 + 3N + 1}{N^2}. 
\end{equation*}
Write 
\begin{equation*}
C'_{1} = \frac{MN}{3 N + 1}, \text{ and } \ C'_{2} = \frac{N^2}{3N + 1}. 
\end{equation*}
Let $K'$ be a $\left( C'_1, M + \frac{M}{N} - 1 \right)$-Cantor set and $L'$ be a $(C'_2, N)$-Cantor set. Then, 
by arguing analogously,  $\tau(K') = M$, $\tau(L') = N$, and 
$K' \cdot L'$ is a disjoint union of $\{0\}$ and countably many closed intervals.
\end{proof}

\begin{rem}\label{remark}
It is immediate from the construction above that if the condition (\ref{intersection}) is satisfied, there exist 
$0$-Cantor sets $K$ and $L$, such that 
\begin{itemize}
\item[(1)] $\tau(K) = M, \ \tau(L) = N$; 
\item[(2)] neither $K$ nor $L$ lies in a complementary domain of the other; 
\item[(3)] $K \cap L$ consists of exactly one element.
\end{itemize}
In fact, \cite{Hunt} and \cite{Kraft} independently showed that the condition (\ref{intersection}) is the optimal estimate 
that guarantees the existence of such $K$ and $L$. See section \ref{65743}. 
\end{rem}

\section{Other cases}\label{5879}

In this section, we consider the cases that we have not yet discussed. 
The proofs are analogous, so we only state the results. Recall that a Cantor set $K$ is a 
$0^{+}_{-}$-Cantor set if $K_{+}, K_{-} \neq \phi$, $\inf K_+ = 0$, and $\inf K_- > 0$. 

\begin{thm}
Suppose that either of the following holds: 
\begin{itemize}
\item[(1)] $K$ is a $0^+$-Cantor set, and $L$ is a $0^+_-$-Cantor set; 
\item[(2)] $K$ is a $0^+$-Cantor set, and $L$ is a $0^{\times}$-Cantor set.
\end{itemize}
Then, if the condition (\ref{3654}) is satisfied, $K \cdot L$ is an interval. 
Furthermore, let $M, N > 0$ be real numbers that satisfy the condition (\ref{thm2}). 
Then, for any $k \geq 2$, there exist Cantor sets $K, L,$ such that $K$ and $L$ satisfy one of the conditions above, 
$\tau(K) = M$, $\tau(L) = N$, and 
$K \cdot L$ is a disjoint union of $k$ closed intervals. 
\end{thm}

\begin{thm}
Suppose that one of the following holds: 
\begin{itemize}
\item[(1)] $K$ is a $0$-Cantor set, and $L$ is a $0^{+}_-$ Cantor set; 
\item[(2)] $K$ is a $0$-Cantor set, and $L$ is a $0^{\times}$-Cantor set; 
\item[(3)] $K$ is a $0^{+}_-$-Cantor set, and $L$ is a $0^{\times}$-Cantor set; 
\item[(4)] $K$ and $L$ are both $0^{+}_-$-Cantor sets. 
\end{itemize}
Then, if the condition (\ref{46578}) is satisfied, $K \cdot L$ is an interval. 
Furthermore, let $M, N > 0$ be real numbers with condition (\ref{thm3}).
Then, there exist Cantor sets $K, L$, such that $K$ and $L$ satisfy 
one of the conditions above, 
$\tau(K) = M$, $\tau(L) = N$, and $K \cdot L$ is a disjoint union of two intervals.
\end{thm}

If $K$ and $L$ are both $0^\times$-Cantor sets, 
the best we can hope for is $K \cdot L$ to become 
a disjoint union of two intervals. 

\begin{thm}
Let $K, L$ be $0^{\times}$-Cantor sets. Then, if the condition (\ref{46578}) is satisfied, 
$K \cdot L$ is a disjoint union of two intervals. 
Furthermore, let $M, N > 0$ be real numbers with condition (\ref{thm3}). 
Then there exist $0^\times$-Cantor sets $K$ and $L$ such that $\tau(K) = M$, $\tau(L) = N$, and 
$K \cdot L$ is a disjoint union of three intervals.
\end{thm}

So far, we have not considered the case that $\min K > 0$ or $\min L > 0$. 
In fact, this turns out to be very simple. We have the following:
\begin{thm}
For any real numbers $M, N > 0$,  
there exist a Cantor set $K$ and a $0^+$-Cantor set $L$ such that $\min K > 0$, $\tau(K) = M$, $\tau(L) = N$, 
and $K \cdot L$ is a disjoint union of $\{0\}$ and countably 
many closed intervals. 
\end{thm} 

\begin{proof}[Outline of the proof]
Let us take sufficiently small $\epsilon > 0$. Let $K$ be a Cantor set such that 
\begin{equation*}
\begin{aligned}
(1)& \ K \subset [1, 1 + \epsilon] \text{ and } \tau(K) = M;  \\
(2)& \ K = K_1 \sqcup K_2, \text{ where } K_1, K_2 \text{ are Cantor sets with sufficiently large thickness}. 
\end{aligned}
\end{equation*}
Let $L$ be the middle $\frac{1}{1 + 2N}$-Cantor set whose convex hull is $[0, 1]$. It is easy to see that $K \cdot L$ is 
the disjoint union of $\{0\}$ and countably many closed intervals. 
\end{proof}

\section{Connection with questions on intersections of two Cantor sets}\label{65743}
In this section, we discuss the connection between questions on 
products of two Cantor sets and questions on intersections of two Cantor sets. 
For any Cantor sets $K$ and $L$, if neither $K$ nor $L$ lies in a complementary domain of the other we say $K$ and $L$ are 
\emph{interleaved}. 
Williams showed the following in \cite{Williams}
(this result was later extended by \cite{Hunt} and \cite{Kraft}, independently): 

\begin{thm}[Theorem 1 of \cite{Williams}]\label{98}
Let $K, L$ be interleaved Cantor sets. Then if $\tau(K), \tau(L) \geq 1 + \sqrt{2}$, 
$K \cap L$ contains infinitely many elements. 
Furthermore, for any $M, N < 1 + \sqrt{2}$ there exist interleaved Cantor sets 
$K$ and $L$ such that $\tau(K) = M$, $\tau(L) = N$, and  
$K \cap L$ consists of exactly one element. 
\end{thm}

\begin{rem}
In fact, Williams showed much more. For example, he also showed that if 
$\tau(K),  \tau(L) > 1 + \sqrt{2}$, 
$K \cap L$ contains a Cantor set. We are not sure whether our method can be applied to prove this statement. 
\end{rem}
To illustrate the connection, we present a completely different proof of Theorem \ref{98} using our method.  
See also Remark \ref{remark}. 

\begin{proof}[outline of the proof of Theorem \ref{98}]
For the sake of simplicity, we only consider the case $\tau(K) = \tau(L) = 1 + \sqrt{2}$. 
Write $C = \frac{1}{\sqrt{2}}$. 

By the Gap Lemma, $K \cap L \neq \phi$. Translating $K$ and $L$ if necessary, 
we can assume that $0 \in K \cap L$. 
Let us only consider the case that $K$ and $L$ are both $0$-Cantor sets. 

Let $\{ U_{n} \} \ (n = 0, 1, \cdots)$ and $\{ V_{n}\} \ (n = 0, 1, \cdots)$ 
be the $\log$-$C$-split gaps of $\widetilde{K_+}$ and the $\log$-$C$-split Cantor sets of $\widetilde{K_+}$, respectively. 
Let $X_{n} \ (n = 1, 2, \cdots)$ be $\log$-$C$-covers of $U_{n} \ (n = 1, 2, \cdots)$. 
Similarly, let $\{ S_{n} \}$ and $\{ T_{n}\}$ 
be the $\log$-$C$-split gaps of $\widetilde{L_{+}}$ and the $\log$-$C$-split Cantor sets of $\widetilde{L_+}$, respectively. 
Let $Y_{n}$ be a $\log$-$C$-cover of $S_{n}$. For the sake of simplicity, we assume that both splits are infinite.
By Lemma \ref{akb} and the Gap Lemma, we have 
\begin{equation*}
\begin{aligned}
(1)& \ \text{ for all $n \in \mathbb{N},$ } |X_n| \geq |U_n| \text{ and }  |Y_n| \geq |S_n|;  \\
(2)& \ \text{ for all $n, m \in \mathbb{N}$}, \ V_n \cap T_m \neq \phi \text{ \ if \  } 
\mathrm{con}(V_n) \cap \mathrm{con}(T_m) \neq \phi; \\
(3)& \ \text{ for all $n, m \in \mathbb{N}$}, \ X_n \cap Y_m \neq \phi \text{ \ if \  } \mathrm{con}(X_n) \cap \mathrm{con}(Y_m) \neq \phi; 
\end{aligned}
\end{equation*}
where $\mathrm{con}(A)$ is the convex hull of a set $A$. 
The claim follows from this.  
\end{proof}

\section{Open problems}\label{47325}
In this section, we state a few questions and open problems that are suggested by the results of this paper.
\begin{enumerate}
\item In \cite{Astels}, the author generalized Theorem \ref{sum_sum_sum} for sums of three or more Cantor sets. 
It is natural to try to extend their results to products of three or more Cantor sets. 
\item It is also natural to consider the product of middle $\alpha$-Cantor set and middle $\beta$-Cantor set instead of general Cantor sets. 
\item Similarly, we can consider products of two dynamically defined Cantor sets instead of general Cantor sets. 
In this case, we believe that the estimate in Theorem \ref{migotoda} will be different; but this question is 
beyond the scope of our paper.  
\end{enumerate}

\section*{Acknowledgements}
The author would like to acknowledge the invaluable contributions of Anton Gorodetski. 
The author would also like to thank the anonymous referees for many helpful suggestions and remarks.

\end{document}